\documentclass[12pt, reqno]{amsart}

\author[P.~Leonetti]{Paolo Leonetti}
\address{Institute of Analysis and Number Theory, Graz University of Technology,  Kopernikusgasse 24/II, 8010 Graz, Austria}
\email{leonetti.paolo@gmail.com}
\urladdr{\url{https://sites.google.com/site/leonettipaolo/}} 

\keywords{Ideal core, ideal cluster point, closed convex hull, Pringsheim limit, $e$-convergence, double sequence, Euler $r$-core.}
\thanks{P.L. was supported by the Austrian Science Fund (FWF), project F5512-N26.}
\subjclass[2010]{Primary: 40A35. Secondary: 54A20, 40A05.}

\title{Characterizations of the Ideal Core}

\usepackage[T1]{fontenc}
\usepackage{amsmath}
\usepackage{amssymb}
\usepackage{amsthm}
\usepackage[left=3.5cm, right=3.5cm, paperheight=11.8in]{geometry}
\usepackage{hyperref}
\usepackage{fancyhdr}
\usepackage{enumitem}
\usepackage{comment}
\usepackage{nicefrac}
\usepackage{bm}
\usepackage{mathrsfs}
\usepackage{graphicx}
\usepackage[utf8]{inputenc}
\usepackage{cancel}
\usepackage{mathtools}

\newcommand{\vertiii}[1]{{\left\vert\kern-0.25ex\left\vert\kern-0.25ex\left\vert #1 
    \right\vert\kern-0.25ex\right\vert\kern-0.25ex\right\vert}}
		
\AtBeginDocument{%
   \def\MR#1{}
}

\newtheorem{thm}{Theorem}[section]
\newtheorem{cor}[thm]{Corollary}%[section]
\newtheorem{lem}[thm]{Lemma}
\newtheorem{prop}[thm]{Proposition}

\theoremstyle{definition} 
\newtheorem{defi}[thm]{Definition}%[section]
\let\olddefi\defi
\renewcommand{\defi}{\olddefi\normalfont}
\newtheorem{example}[thm]{Example}
\let\oldexample\example
\renewcommand{\example}{\oldexample\normalfont}
\newtheorem{rmk}[thm]{Remark}
\let\oldrmk\rmk
\renewcommand{\rmk}{\oldrmk\normalfont}

\hypersetup{
    pdftitle={{\textsc{Characterizations of the Ideal Core}}},
    pdfauthor={Paolo Leonetti},
    pdfmenubar=false,
    pdffitwindow=true,
    pdfstartview=FitH,
    colorlinks=true,
    linkcolor=blue,
    citecolor=green,
    urlcolor=cyan
}

\providecommand{\MR}[1]{}
\providecommand{\bysame}{\leavevmode\hbox to3em{\hrulefill}\thinspace}
\providecommand{\MR}{\relax\ifhmode\unskip\space\fi MR }

\providecommand{\href}[2]{#2}

\begin{document}

\maketitle
\thispagestyle{empty}

\begin{abstract}
Given an ideal $\mathcal{I}$ on $\omega$ and a sequence $x$ in a topological vector space, we let the $\mathcal{I}$-core of $x$ be the 
least closed convex set containing $\{x_n: n \notin I\}$ for all $I \in \mathcal{I}$. We show two characterizations of the $\mathcal{I}$-core. This implies that the $\mathcal{I}$-core of a bounded sequence in $\mathbf{R}^k$ is simply the convex hull 
of its $\mathcal{I}$-cluster points. As applications, we simplify and extend several results in the context of Pringsheim-convergence and $e$-convergence of double sequences. 
\end{abstract}

%%%%%%%%%%%%%%%%%%%%%%%%%%%%%%%%%%%%%%%%%%%%%%%%%%%%%%%%%%%%%%%%%%%%%%%%%%%%

\section{Introduction}\label{sec:intro}

Let $\mathcal{I}$ be an ideal on the positive integers $\omega$, that is, a family of subsets of $\omega$ closed under taking finite unions and subsets. It is also assumed that $\mathcal{I}$ contains the family of finite sets $\mathrm{Fin}$ and that $\mathcal{I}$ is not the entire power set $\mathcal{P}(\omega)$. Let $\mathcal{I}^\star$ be the dual filter, that is, the family of $A$ such that $A^c \in \mathcal{I}$. Let $\mathcal{Z}$ be the ideal of asymptotic density zero sets, i.e., $\mathcal{Z}=\{A\subseteq \omega: \lim_n \frac{1}{n}|A\cap [1,n]|=0\}$.

Given a sequence $x$ taking values in a topological space $X$, denote by $\Gamma_x(\mathcal{I})$ the set of $\mathcal{I}$-cluster points of $x$, that is, the set of all $\ell \in X$ such that $\{n \in \omega: x_n \in U\} \notin \mathcal{I}$ for all neighborhoods $U$ of $\ell$. $\mathcal{I}$-cluster points are usually called statistical cluster points if $\mathcal{I}=\mathcal{Z}$, see \cite{MR1181163}. It is known that $\Gamma_x(\mathcal{I})$ is closed. In particular, $\Gamma_x(\mathcal{I})$ is compact for all real bounded sequences $x$, cf. \cite{MR3883171, Leo17, LMxyz} for basic facts and characterizations of $\mathcal{I}$-cluster points.

The classical \emph{Knopp core} of a complex sequence $x$ is defined by 
\begin{equation}\label{eq:originalknopp}
\mathrm{K}\text{-}\mathrm{core}\{x\}:=\bigcap_{k \in \omega} \overline{\mathrm{co}}\,\{x_n: n \ge k\},
\end{equation}
where $\overline{\mathrm{co}}\,A$ denotes the closed convex hull of $A\subseteq \mathbf{C}$, cf. \cite[Section 1]{MR1441452} and \cite{MR0142952}. It is shown in \cite{MR0487146} that if $x$ is, in addition, bounded, then the Knopp core of $x$ coincides with
\begin{equation}\label{eq:firstcharacterizationold}
\bigcap_{y \in \mathbf{C}}{\{z \in \mathbf{C}: |z-y| \le \limsup_{n\to \infty}|x_n-z|\}}
\end{equation}

It is clear that the Knopp core of a real bounded sequence $x$ is the compact interval $[\liminf_n x_n, \limsup_n x_n]$. Then, the \emph{statistical core} of a real sequence $x$ has been introduced in \cite[Section 3]{MR1416085} as the closed interval $[a,b]$, where $a$ and $b$ are the statistical limit inferior and statistical limit superior of $x$, that is, $\inf \Gamma_x(\mathcal{Z})$ and $\sup \Gamma_x(\mathcal{Z})$, respectively. This has been soon extended to complex sequences $x$ in \cite{MR1441452}, the analogue of \eqref{eq:originalknopp} being
\begin{equation}\label{eq:statisticalcoreolddefinition}
\mathrm{st}\text{-}\mathrm{core}\{x\}:=\bigcap_{H \in \mathscr{H}(x)}H,
\end{equation}
where $\mathscr{H}(x)$ is the collection of closed half-planes which contain almost all $x_k$, that is, except a set of $x_k$ for which the index set has asymptotic density zero. 
Right after, the authors prove the statistical analogue of \eqref{eq:firstcharacterizationold} for complex bounded sequences $x$, where they replace $\limsup_{n}|x_n-z|$ with the statistical limit superior of the real sequence $|x_n-z|$.

The last extension for real bounded sequences $x$ is due to Connor, see \cite[p. 410]{MR1734462} 
and \cite[Section 2]{MR2241135}: 
given a ``density'' $\mu$, that is, a finitely additive diffuse probability measure defined on a field of subsets of $\omega$, 
%the $\mathcal{I}$-core of $x$ 
the $\mu$\emph{-statistical core} of $x$ 
is the compact interval 
\begin{equation}\label{eq:connotstmu}
\mathrm{st}_\mu\text{-}\mathrm{core}\{x\}:=[\min \Gamma_x(\mathcal{I}_\mu),\max \Gamma_x(\mathcal{I}_\mu)],
\end{equation}
where $\mathcal{I}_\mu$ is the kernel of $\mu$. On the other hand, for every ideal $\mathcal{I}$, there exists a density $\mu$ such that $\mathcal{I}=\mathcal{I}_\mu$ (it is sufficient to let $\mu$ be the characteristic function of $\mathcal{I}^\star$, defined on $\mathcal{I}\cup \mathcal{I}^\star$), hence there is no further need to speak about densities.

The aim of this article is to provide a general notion of $\mathcal{I}$-core of a sequence in a topological vector space, unify the above characterizations, and establish its relationship with the set of $\mathcal{I}$-cluster points. In particular, we show that if $x$ is a bounded sequence in $\mathbf{R}^k$ then its $\mathcal{I}$-core is simply the convex hull of $\Gamma_x(\mathcal{I})$. As an application, we simplify and extend several results in the context of Pringsheim-convergence and $e$-convergence of double sequences. 

%We provide a general notion of ideal core of a sequence in a topological vector space which include and unify many variants proposed in the literature. Also, we provide characterizations which establish the relationship with the set of ideal cluster points. In particular, we show that the ideal core of a bounded sequence in $\mathbf{R}^k$ is simply the convex hull of the set of ideal cluster points. As an application, we simplify and extend several results in the context of Pringsheim-convergence and e-convergence of double sequences. 

%%%%%%%%%%%%%%%%%%%%%%%%%%%%%%%%%%%%%%%%%%%%%%%%%%%%%%%%%%%%%%%%%%%%%%%%%%%%%%%%%%%%
\section{Ideal Core and Main Results}

We start with our main definition:
\begin{defi}\label{def:maindefinitionIcore}
Given an ideal $\mathcal{I}$ on $\omega$, the $\mathcal{I}$-core of a sequence $x$ taking values in a topological vector space $X$ is 
\begin{equation}\label{eq:maindefinitionIcoreTVS}
\mathrm{core}_x(\mathcal{I}):=\bigcap_{E \in \mathcal{I}^\star}\overline{\mathrm{co}}\,\{x_n:n\in E\}.
\end{equation}
\end{defi}
In other words, the $\mathcal{I}$-core of $x$ is the least closed convex set containing $\overline{\mathrm{co}}\,\{x_n:n\in E\}$ for all $E \in \mathcal{I}^\star$. 
%biggest closed convex set contained in all $\overline{\mathrm{co}}\,\{x_n:n\in E\}$, with $E \in \mathcal{I}^\star$. 
(Hereafter, $\mathrm{co}\,A$ and $\overline{\mathrm{co}}\,A$ stand for the convex hull and closed convex hull of $A\subseteq X$, respectively.) 

Note that if $x$ is a complex sequence and $\mathcal{I}=\mathrm{Fin}$ then $\mathrm{core}_x(\mathcal{I})$ is the Knopp core defined in \eqref{eq:originalknopp}. Also, the statistical core defined in \eqref{eq:statisticalcoreolddefinition} can be rewritten as 
$$
\bigcap_{E \in \mathcal{Z}^\star} \,\bigcap_{\substack{H \text{ half-closed plane},\\ \{x_n:\, n \in E\}\,\subseteq \, H}}H,
$$
and it is known that, for each $E \in \mathcal{Z}^\star$, the inner intersection coincides with $\overline{\mathrm{co}}\,\{x_n: n \in E\}$, hence $\mathrm{st}\text{-}\mathrm{core}\{x\}=\mathrm{core}_x(\mathcal{Z})$. Lastly, if $x$ is a real bounded sequence, then the $\mu$-statistical core defined in \eqref{eq:connotstmu} corresponds to $\mathrm{core}_x(\mathcal{I}_\mu)$, as it follows from Corollary \ref{corIcounded} below.

Our main result follows:
\begin{thm}\label{thm:charactcore}
Let 
$x$ be a sequence in a first countable locally convex topological vector space $X$ 
such that the closure of $\{x_n:n \in E\}$ is compact for some $E \in \mathcal{I}^\star$. Then $$\mathrm{core}_x(\mathcal{I})=\overline{\mathrm{co}}\,\Gamma_x(\mathcal{I}).$$
\end{thm}

Note that the hypothesis that the closure of $\{x_n:n \in E\}$ is compact for some $E \in \mathcal{I}^\star$ cannot be omitted from Theorem \ref{thm:charactcore}. Indeed, consider the following example: let $x$ be the real sequence defined by $x_n=(-1)^n n$ for all $n$. Then $\mathrm{core}_x(\mathcal{Z})=\mathbf{R}$ and $\overline{\mathrm{co}}\,\Gamma_x(\mathcal{Z})=\emptyset$.

We recall that a sequence $x$ in $\mathbf{R}^k$ is said to be $\mathcal{I}$-bounded whenever there exists a constant $c$ such that $\{n \in \omega: \|x_n\|>c\} \in \mathcal{I}$. Accordingly:
\begin{cor}\label{corIcounded}
Let $x$ be an $\mathcal{I}$-bounded sequence in $\mathbf{R}^k$. Then 
$$
\mathrm{core}_x(\mathcal{I})=\mathrm{co}\,\Gamma_x(\mathcal{I}).
$$
Hence $\ell$ belongs to the $\mathcal{I}$-core if and only if $\ell \in \mathrm{co}\{z_1,\ldots,z_{k+1}\}$ for some $\mathcal{I}$-cluster points $z_1,\ldots,z_{k+1}$.
\end{cor}

Moreover, 
we have the following representation:
\begin{cor}\label{cor:charactcorekrein} 
Let $x$ be a sequence in a Banach space such that $\{n: x_n \notin K\} \in \mathcal{I}$ for some compact $K$. 
Then, $\ell \in \mathrm{core}_x(\mathcal{I})$ if and only if there exists a Borel regular probability measure $\mu$ supported on $\Gamma_x(\mathcal{I})$ such that
$$
\ell=\int_{\Gamma_x(\mathcal{I})} x\,\mu(\mathrm{d}x).
$$
\end{cor}
%We remark that, since $\Gamma_x(\mathcal{I})$ is compact, then $E\subseteq \Gamma_x(\mathcal{I})$ by Milman's theorem \cite[Theorem 3.25]{MR1157815}.

Lastly, if $X$ is an Hilbert space, we obtain the following extension of \eqref{eq:firstcharacterizationold}:
\begin{thm}\label{thm:anothercharactIcoremain}
%Let $\mathcal{I}$ be an ideal and 
Let 
$x$ be a sequence in a Hilbert space $X$ such that the closure of $\{x_n:n \in E\}$ is compact for some $E \in \mathcal{I}^\star$. Then 
$$%\begin{equation}\label{eq:claimsecondcharacterizationcore}
\mathrm{core}_x(\mathcal{I})=\bigcap_{y \in X}F_{x,y}(\mathcal{I})
$$%\end{equation}
where $F_{x,y}(\mathcal{I}):=\{h \in X: \|h-y\|\le \max \Gamma_{\|x_n-y\|}(\mathcal{I})\}$ for all $y \in X$.
%, and $\|\cdot\|$ is the norm induced by the inner product.
\end{thm}
%Identity \eqref{eq:firstcharacterizationold} corresponds to the case $X=\mathbf{C}$ and $\mathcal{I}=\mathrm{Fin}$.
%The answer is positive if $X=\mathbf{C}$ and $\mathcal{I}=\mathrm{Fin}$, see \cite{MR0487146}. \textcolor{red}{[Delete]}
We leave as open question for the interested reader to establish whether Theorem \ref{thm:anothercharactIcoremain} holds in Banach space.

Proofs follow in Section \ref{sec:proofs}. Applications are given in Section \ref{sec:applications}.

%%%%%%%%%%%%%%%%%%%%%%%%%%%%%%%%%%%%%%%%%%%%%%%%%%%%%%%%%%%%%%%%%%

\section{Proofs}\label{sec:proofs}%of main results

Let us start with a preliminary lemma.
\begin{lem}\label{lem:firstinclusioncore}
%Let $\mathcal{I}$ be an ideal and 
Let 
$x$ be a sequence in a topological vector space with a countable local base at $0$. Then $\overline{\mathrm{co}}\,\Gamma_x(\mathcal{I})\subseteq \mathrm{core}_x(\mathcal{I})$.
\end{lem}
\begin{proof}
Since $X$ is first countable, it follows by \cite[Theorem 4.2]{LMxyz} that 
$$
\Gamma_x(\mathcal{I})=\bigcap_{E \in \mathcal{I}^\star} \overline{\{x_n: n\in E\}}.
$$
In addition, $\overline{\{x_n:n \in E\}}\subseteq \overline{\mathrm{co}}\,\{x_n:n\in E\}$ for all $E \in \mathcal{I}^\star$, hence $\Gamma_x(\mathcal{I})\subseteq \mathrm{core}_x(\mathcal{I})$. Therefore $\overline{\mathrm{co}}\,\Gamma_x(\mathcal{I})\subseteq \overline{\mathrm{co}}\,\mathrm{core}_x(\mathcal{I})=\mathrm{core}_x(\mathcal{I})$.
\end{proof}

We can proceed to the proof of our main result.
\begin{proof}[Proof of Theorem \ref{thm:charactcore}]
The inclusion $\overline{\mathrm{co}}\,\Gamma_x(\mathcal{I})\subseteq \mathrm{core}_x(\mathcal{I})$ follows by Lemma \ref{lem:firstinclusioncore}. 
At this point, we have only to show that
\begin{equation}\label{eq:claimreverseinclusioncore}
\bigcap_{E \in \mathcal{I}^\star}\overline{\mathrm{co}}\,\{x_n:n\in E\} \subseteq \overline{\mathrm{co}}\,\Gamma_x(\mathcal{I}).%.\left(\bigcap_{E \in \mathcal{I}^\star} \overline{\{x_n: n\in E\}}\right).
\end{equation}
To this aim, fix $\ell \in X\setminus \overline{\mathrm{co}}\,\Gamma_x(\mathcal{I})$. By the Hahn--Banach separation theorem \cite[Theorem 3.4]{MR1157815}, there exist a continuous linear functional $T: X\to \mathbf{R}$ and $r \in \mathbf{R}$ such that $T(\ell)<r<T(\eta)$ for all $\eta \in \overline{\mathrm{co}}\,\Gamma_x(\mathcal{I})$. It follows that $G:=T^{-1}((r,\infty))$ is a convex open set such that $\ell \notin G$ and, thanks to \cite[Theorem 4.2]{LMxyz}, 
$$
\bigcap_{E \in \mathcal{I}^\star} \overline{\{x_n: n\in E\}}=\Gamma_x(\mathcal{I})\subseteq \overline{\mathrm{co}}\,\Gamma_x(\mathcal{I})\subseteq G.
$$
In addition, $F:=T^{-1}([r,\infty))$ is a closed convex set such that $\ell \notin F$ and $G\subseteq F$.

Considering that there exists $E \in \mathcal{I}^\star$ such that the closure of $\{x_n: n\in E\}$ is compact, it follows by \cite[Corollary 3.1.5]{MR1039321} that there exist $E_1,\ldots,E_k \in \mathcal{I}^\star$ such that $\bigcap_{i=1}^k \overline{\{x_n: n\in E_i\}}\subseteq G$. Setting $E_0:=E_1\cap \cdots \cap E_k \in \mathcal{I}^\star$, we obtain 
$$
\overline{\{x_n: n \in E_0\}}\subseteq \bigcap_{i=1}^k\overline{\{x_n: n \in E_i\}}\subseteq G\subseteq F.
$$

To sum up, for every $\ell \notin\overline{\mathrm{co}}\,\Gamma_x(\mathcal{I})$ there exists $E_0=E_0(\ell) \in \mathcal{I}^\star$ such that the closed convex hull of $\overline{\{x_n: n \in E_0\}}$ is contained in $F$; in particular,
$$
\overline{\mathrm{co}}\, \{x_n: n \in E_0\} \subseteq \overline{\mathrm{co}}\, \overline{\{x_n: n \in E_0\}} \subseteq F\subseteq X\setminus \{\ell\},
$$ 
hence $\ell \notin \overline{\mathrm{co}}\, \{x_n: n \in E_0\}$. This implies the inclusion \eqref{eq:claimreverseinclusioncore}. 
\end{proof}

\begin{proof}[Proof of Corollary \ref{corIcounded}]
It follows by the standing hypothesis that there exists $E \in \mathcal{I}^\star$ such that the closure of $\{x_n:n \in E\}$ is compact. In addition, it is known that $\overline{\mathrm{co}}\,A=\mathrm{co}\, \overline{A}$ for all compact sets $A\subseteq \mathbf{R}^k$. Therefore, thanks to Theorem \ref{thm:charactcore}, we conclude that $\mathrm{core}_x(\mathcal{I})=\overline{\mathrm{co}}\,\Gamma_x(\mathcal{I})=\mathrm{co}\,\overline{\Gamma_x(\mathcal{I})}=\mathrm{co}\,\Gamma_x(\mathcal{I})$. The second claim follows by Caratheodory's theorem, cf. \cite[p.73]{MR1157815}.
\end{proof}

%%%%%%%%%%%%%%%%%%%%%%%%%%%%%%%%%%%%%%%%%%%

%%%%%%%%%%%%%%%%%%%%%%%%%%%%%%%%%%%%%%%%%%%%%%%%%%%%%

In the following proof, we denote by $\mathrm{ex}\,A$ the set of extreme points of a convex set $A$ in a vector space, that is, the set of all $a\in A$ such that if $a=tx+(1-t)y$ for some $t \in (0,1)$ and $x,y \in A$ then $x=y=a$.
\begin{proof}[Proof of Corollary \ref{cor:charactcorekrein}]
Recall that a Banach space is a locally convex first countable topological vector space and its continuous dual separates points. Now, if $A:=\{n: x_n \notin K\} \in \mathcal{I}$ then $A^c \in \mathcal{I}^\star$, hence the closure of $\{x_n: n \in A^c\}$ is compact. Moreover, $\overline{\mathrm{co}}\,K$ is compact by \cite[Theorem 3.20.c]{MR1157815} and convex by \cite[Theorem 1.13.d]{MR1157815}.

Considering that $\mathrm{core}_x(\mathcal{I})\subseteq \overline{\mathrm{co}}\,K$ and that $\Gamma_x(\mathcal{I})$ is nonempty by \cite[Lemma 3.1.vi]{LMxyz} and contained in the $\mathcal{I}$-core of $x$ by Lemma \ref{lem:firstinclusioncore}, it follows that $\mathrm{core}_x(\mathcal{I})$ is a nonempty convex compact set. Therefore, by the Krein--Milman theorem \cite[Theorem 3.23]{MR1157815} and Theorem \ref{thm:charactcore}, we obtain that
\begin{displaymath}
\mathrm{core}_x(\mathcal{I})=\overline{\mathrm{co}}\,(\mathrm{ex}\,\mathrm{core}_x(\mathcal{I}))=\overline{\mathrm{co}}\,E,
\end{displaymath}
where $E:=\mathrm{ex}\,\overline{\mathrm{co}}\,\Gamma_x(\mathcal{I})$. Lastly, by the Choquet's theorem \cite[Exercise 25, p.89 and p.402]{MR1157815}, for each $\ell \in \mathrm{core}_x(\mathcal{I})$ there exists a Borel regular probability measure $\mu$ supported on $E$ such that
$$
\ell=\int_{E} x\,\mu(\mathrm{d}x).
$$
On the other hand, $E$ is contained in the compact set $\Gamma_x(\mathcal{I})$, thanks to Milman's theorem \cite[Theorem 3.25]{MR1157815}.

Conversely, if $\mu$ is a Borel regular probability measure supported on $\Gamma_x(\mathcal{I})$ then 
$$
\int_{\Gamma_x(\mathcal{I})} x\, \mu(\mathrm{d}x) \in \overline{\mathrm{co}}\,\Gamma_x(\mathcal{I})=\mathrm{core}_x(\mathcal{I}),
$$ 
see \cite[Theorem 3.27]{MR1157815}.
\end{proof}

%%%%%%%%%%%%%%%%%%%%%%%%%%%%%%%%%%%%%%%%%%%%%%%%%%%%%%%%%%%%%%%%%%%%%%%%%%%%%%%%%%%%%%%%%%%%%

Recall that, according to \cite[Theorem 1.24]{MR1157815}, a first countable topological vector space is locally convex if and only if its topology is compatible with an invariant metric for which the open balls are balanced and convex. Denoting by $d$ a compatible metric of this type, we have:

\begin{lem}\label{lem:anothercharactIcore}
With the same hypotheses of Theorem \ref{thm:charactcore} and the notation of Theorem \ref{thm:anothercharactIcoremain}, it holds
$$%\begin{equation}\label{eq:claimsecondcharacterizationcore}
\mathrm{core}_x(\mathcal{I})\subseteq \bigcap_{y \in X}F_{x,y}(\mathcal{I}),
$$%\end{equation}
where $F_{x,y}(\mathcal{I}):=\{h \in X: d(y,h)\le \max \Gamma_{d(x_n,y)}(\mathcal{I})\}$ for all $y \in X$. 
\end{lem}
\begin{proof}%[Proof of Theorem \ref{thm:anothercharactIcore}]
First, let us show that $\Gamma_{d(x_n,y)}(\mathcal{I})$ admits a maximum so that each $F_{x,y}(\mathcal{I})$ is well defined. To this aim, fix $y \in X$. By hypothesis there exists $E \in \mathcal{I}^\star$ such that the closure of $\{x_n: n \in E\}$, hereafter denoted by $K$, is compact. Then fix some $\kappa \in K$ and define the sequence $z=(z_n)$ by $z_n=x_n$ if $x \in E$ and $z_n=\kappa$ otherwise. Thanks to \cite[Lemma 3.1.v]{LMxyz}, we have $\Gamma_{d(x_n,y)}(\mathcal{I})=\Gamma_{d(z_n,y)}(\mathcal{I})$. Hence, let us suppose without loss of generality that $x_n \in K$ for all $n\in \omega$. 

Considering that the section $d(\cdot,y)$ is continuous and $K \cup \{y\}$ is compact, the real sequence $d(x_n,y)$ is bounded by a constant $\mathrm{c}=\mathrm{c}(x,y)$, hence $\Gamma_{d(x_n,y)}(\mathcal{I})$ is compact and admits a maximum, let us say $\mathrm{m}=\mathrm{m}(x,y)$. 
Note also that $F_{x,y}(\mathcal{I})$ is convex since can be rewritten as $\bigcap_{n}\left\{h: d(y,h)< \mathrm{m}+\nicefrac{1}{n}\right\}$ and each open ball is convex.

At this point, fix also $\ell \in \Gamma_x(\mathcal{I})$. We claim that $\ell \in F_{x,y}(\mathcal{I})$. Indeed, suppose for the sake of contradiction that $\ell\notin F_{x,y}(\mathcal{I})$, that is, $d(\ell,y)>\mathrm{m}$ and set $\varepsilon:=\frac{1}{2}\left(d(\ell,y)-\mathrm{m}\right)>0$. Since $\ell$ is an $\mathcal{I}$-cluster point and $d(x_n,y) \ge d(\ell,y)-d(x_n,\ell)$ for all $n \in \omega$, we obtain
$$
\{n \in \omega: d(x_n,y) \in [d(\ell,y)-\varepsilon,\mathrm{c}]\} \supseteq \{n \in \omega: d(x_n,\ell)\le \varepsilon\} \notin \mathcal{I}.
$$
Considering that $[d(\ell,y)-\varepsilon,\mathrm{c}]$ is compact, it follows by \cite[Lemma 3.1.vi]{LMxyz} that it contains an $\mathcal{I}$-cluster point of the sequence $d(x_n,y)$ (which is, in particular, bigger than $\mathrm{m}$), contradicting that $\Gamma_{d(x_n,y)}(\mathcal{I}) \subseteq [0,\mathrm{m}]$. Therefore $\Gamma_x(\mathcal{I}) \subseteq \bigcap_{y \in X}F_{x,y}(\mathcal{I})$. Since the latter intersection is closed and convex, it follows by Theorem \ref{thm:charactcore} that
$$
\mathrm{core}_x(\mathcal{I})=\overline{\mathrm{co}}\,\Gamma_x(\mathcal{I})\subseteq \bigcap_{y \in X}F_{x,y}(\mathcal{I}),
$$
which completes the proof.
\end{proof}

\begin{proof}[Proof of Theorem \ref{thm:anothercharactIcoremain}]
The inclusion $\mathrm{core}_x(\mathcal{I})\subseteq \bigcap_{y \in X}F_{x,y}(\mathcal{I})$ follows by Lemma \ref{lem:anothercharactIcore}. Conversely, with the same notation of the proof of Lemma \ref{lem:anothercharactIcore}, we have to prove that 
\begin{equation}\label{eq:converseinequality}
\bigcap_{y \in X}F_{x,y}(\mathcal{I}) \subseteq \bigcap_{E \in \mathcal{I}^\star} \overline{\mathrm{co}}\,\{x_n: n \in E\}.
\end{equation}

To this aim, fix $E \in \mathcal{I}^\star$ and $\ell\notin \overline{\mathrm{co}}\,\{x_n: n \in E\}$. Now since the underlying metric is translation-invariant, we can assume without loss of generality that $\ell=0$ and $0\notin K$. Then we need to show that there exists $y \in X$ such that $\ell\notin F_{x,y}(\mathcal{I})$, that is, $d(0,y)=\|y\|>\mathrm{m}(x,y)$. By the Hahn--Banach separation theorem \cite[Theorem 3.4]{MR1157815}, there exist a continuous linear functional $T: X\to \mathbf{R}$ and $r \in \mathbf{R}$ such that $0=T(0)<r<T(k)$ for all $k \in K$. Then 
$$
H:=\{h \in X: T(h)=r\}
$$ 
is a closed convex hyperplane separating $0$ and $K$. By \cite[Theorem 12.3]{MR1157815} there exists a unique $h_0 \in H$ with minimal norm. In addition, $K$ is bounded, hence the projection of $K$ on $H$ is bounded set, i.e., it is contained in an open ball $B$ of center $h_0$ and radius $\rho$. Since $h_0$ is orthogonal to $h-h_0$ for all $h \in H$, it follows that $nh_0$ is orthogonal to $b-h_0$ for all $n \in \omega$ and $b \in \partial B$. 
It follows that 
$$
\|nh_0-b\|^2=\|nh_0-h_0\|^2+\|h_0-b\|^2=(n-1)^2\|h_0\|^2+\rho^2.
$$
In particular, $\|nh_0-b\|<
%(n-\frac{1}{2})\|h_0\|<
\|nh_0\|$ for all sufficiently large $n$ and $b \in \partial B$. Therefore there exists $n_0\in \omega$ such that the closed ball with center $y:=n_0h_0$ and radius $\sqrt{(n_0-1)^2\|h_0\|^2+\rho^2}$ contains $K$, hence also $\overline{\mathrm{co}}\,\overline{\{x_n: n \in E\}}$, and does not contain the origin. Therefore inclusion \eqref{eq:converseinequality} follows.
\end{proof}

\section{Applications}\label{sec:applications}

\subsection{Ideal convergence} We recall that a sequence $x$ taking values in a topological space is said to be $\mathcal{I}$-convergent to $\ell$, shortened with $x_n \to_{\mathcal{I}} \ell$, provided that $\{n \in \omega: x_n \notin U\} \in \mathcal{I}$ for all neighborhoods $U$ of $\ell$. If $x$ is contained modulo $\mathcal{I}$ in a compact set, we have the following characterization, see \cite[Corollary 3.4]{LMxyz}:

\begin{lem}\label{lema:LMxyz}
Let $x$ 
be a sequence in first countable space such that $\{n: x_n \notin K\} \in \mathcal{I}$ for some compact $K$. Then $x_n \to_{\mathcal{I}} \ell$ if and only if $\Gamma_x(\mathcal{I})=\{\ell\}$.
\end{lem}

Thus, as an immediate consequence of Theorem \ref{thm:charactcore} and Lemma \ref{lema:LMxyz}, we obtain:
\begin{prop}\label{cor:uniquecore}
Let 
$x$ be a sequence in a locally convex first countable topological vector space $X$ such that 
$\{n: x_n \notin K\} \in \mathcal{I}$ for some compact $K$. 
Then $x_n \to_{\mathcal{I}} \ell$ if and only if $\mathrm{core}_x(\mathcal{I})=\{\ell\}$.
\end{prop}
This has been obtained in \cite[Theorem 3]{MR1416085} for the case $X=\mathbf{R}$ and $\mathcal{I}=\mathcal{Z}$ (in such case, 
$\min \Gamma_x(\mathcal{Z})$ and $\max \Gamma_x(\mathcal{Z})$ 
are usually called statistical limit inferior and statistical limit superior, respectively); for the case of real double sequences see \cite[Theorem 2.6]{MR2209588} and \cite[Theorem 2.6]{MR3267184}, cf. Sections below.

Note that the hypothesis $\{n: x_n \notin K\} \in \mathcal{I}$ for some compact $K$ cannot be removed. Indeed, let $x=(u_1,-u_1,u_2,-u_2,\ldots)$ where $u_k$ is the $k$-th unit vector in $\ell_\infty$ with the supremum norm. Then $x$ is not convergent and, on the other hand, its Knopp core is $\{0\}$. %%%%VEDI ESEMPIO SIMONS P.1 THE CORE
%%%%%%%%%%%%%%%%%%%%%%%%%%%%%%%%%%%%%%%%%%%%%%%%%%%%%%%%%%%%%%%%%%%

\subsection{Pringsheim core and statistical Pringsheim core} A real double sequence $x=(x_{n,m}: n,m \in\omega)$ has \emph{Pringsheim limit} $\ell$ provided that for every $\varepsilon>0$ there exists $k \in \omega$ such that $|x_{n,m}-\ell|<\varepsilon$ for all $n,m\ge k$. Identifying ideals on countable sets with ideals on $\omega$ through a fixed bijection, it is easily seen that this is equivalent to $x \to_{\mathcal{I}_{\mathrm{P}}} \ell$, where $\mathcal{I}_{\mathrm{P}}$ is the ideal defined by
\begin{equation}\label{eq:pringsideal}
\mathcal{I}_{\mathrm{P}}:=\left\{A\subseteq \omega \times \omega: \limsup_{n\to \infty}\, \sup \left\{k: (n,k) \in A\right\}<\infty\right\}.
\end{equation}
Equivalently, $\mathcal{I}_{\mathrm{P}}$ is the ideal on $\omega \times \omega$ containing the complements of $[n,\infty)\times [n,\infty)$ for all $n \in \omega$, i.e., the smallest ideal containing vertical lines and horizontal lines.

Then, Patterson introduced in \cite{MR1733279} the \emph{P-core} of a real double sequence $x$ as the least closed convex set which includes all the points $x_{n,m}$, for $n,m\ge k$. This coincides with Definition \ref{def:maindefinitionIcore} for $X=\mathbf{R}$ and $\mathcal{I}=\mathcal{I}_{\mathrm{P}}$. 
Hence, it follows by Corollary \ref{corIcounded} that:
\begin{prop}\label{cor:pringhsheim}
The P-core of a bounded double sequence $x$ with values in $\mathbf{R}^k$ is the convex hull of $\Gamma_x(\mathcal{I}_{\mathrm{P}})$.
\end{prop}

At this point, for each $n,m \in \omega$, let $\mu_{n,m}$ be the uniform probability measure on $\{1,\ldots,n\}\times \{1,\ldots,m\}$ and define the ideal
$$
\mathcal{Z}_{\mathrm{P}}:=\left\{A\subseteq \omega\times \omega: \mu_{n,m}(A) \to_{\mathcal{I}_{\mathrm{P}}} 0\right\}.
$$
Note that $\mathcal{I}_{\mathrm{P}}\subseteq \mathcal{Z}_{\mathrm{P}}$. The notion of convergence of real double sequences $(x_{n,m})$ with respect to the ideal $\mathcal{Z}_{\mathrm{P}}$ has been recently introduced in \cite{MR2019757};  
here, it has been simply defined ``statistical convergence of double sequences''. 

In this regard, the \emph{statistical core} (or, more properly, \emph{statistical P-core}) of a real bounded double sequence $(x_{n,m})$ has been defined in \cite[Definition 3.1]{MR2209588} as the compact interval 
$$
[\min \Gamma_x(\mathcal{Z}_{\mathrm{P}}),\max \Gamma_x(\mathcal{Z}_{\mathrm{P}})].
$$
This coincides with $\mathrm{core}_x(\mathcal{Z}_{\mathrm{P}})$, thanks to Corollary \ref{corIcounded}.

%%%%%%%%%%%%%%%%%%%%%%%%%%%%%%%%%%%%%%%%%%%%%%%%%%%%%%%%%%%%%%%%%%%%%%%%%%%%

\subsection{$e$-core} 
Given ideals $\mathcal{I},\mathcal{J}$ on $\omega$, we let $\mathcal{I}\times \mathcal{J}$ be their \emph{Fubini product}, i.e., the ideal of all sets $A\subseteq \omega \times \omega$ such that 
$$
\{k \in \omega: \{n\in \omega: (k,n) \in A\} \notin \mathcal{J}\} \in \mathcal{I},
$$
cf. \cite[p.8]{MR1711328}. Accordingly, note that the ideal $\mathcal{I}_{\mathrm{P}}$ defined in \eqref{eq:pringsideal} can be written as 
$$
\{A\cup B^\prime: A,B \in \mathrm{Fin} \times \emptyset\},
$$
where $B^\prime:=\{(n,k) \in \omega \times \omega: (k,n) \in B\}$.

A new notion of convergence for double sequences, called $e$\emph{-convergence}, has been studied in \cite[Section 1]{MR3267184}: a real double sequence $x$ is $e$-convergent to $\ell \in \mathbf{R}$ provided that for every $\varepsilon>0$ there exists $m_0 \in \omega$ such that for each $m\ge m_0$ there is $n_0 \in \omega$ for which $|x_{n,m}-\ell| <\varepsilon$ for all $n\ge n_0$. 
%
%Similarly to P-convergence, it 
It is easily seen that $x$ is $e$-convergent to $\ell$ if and only if 
$$
x_{n,m} \to_{\mathcal{I}_{\mathrm{e}}} \ell\,.
$$
Here, the ideal $\mathcal{I}_{\mathrm{e}}$ stands for the transpose of the Fubini product $\mathrm{Fin}\times \mathrm{Fin}$, i.e.,
$$
\mathcal{I}_{\mathrm{e}}:=\{A\subseteq \omega \times \omega: A^T \in \mathrm{Fin}\times \mathrm{Fin}\},
$$
where $A^T:=\{(n,m) \in \omega \times \omega: (m,n) \in A\}$. Notice that $\mathcal{I}_{\mathrm{P}} \subseteq \mathcal{I}_{\mathrm{e}}$, hence $\Gamma_x(\mathcal{I}_{\mathrm{e}}) \subseteq \Gamma_x(\mathcal{I}_{\mathrm{P}})$, 
 %by \cite[Lemma 3.1.(i)]{LMxyz}, 
%and $\mathrm{core}_x(\mathrm{Fin} \times \mathrm{Fin}) \subseteq \mathrm{core}_x(\mathcal{I}_{\mathrm{P}})$, 
which is \cite[Theorem 2.8]{MR3267184}. %In addition, if $y$ is another real double sequence such that $x-y$ is $e$-convergent to $0$, then 

Similarly to the P-core, the $e$\emph{-core} of a real double sequence $x$ is defined in \cite[Definition 2.10]{MR3267184} to be the closed interval 
$
[\inf \Gamma_x(\mathcal{I}_{\mathrm{e}}), \sup \Gamma_x(\mathcal{I}_{\mathrm{e}})].
$ 
Again, thanks to Corollary \ref{corIcounded}, this coincides with $\mathrm{core}_x(\mathcal{I}_{\mathrm{e}})$. 

The following result has been obtained in \cite[Theorem 2.11]{MR3267184} for the case $\mathcal{I}=\mathcal{I}_{\mathrm{e}}$ and bounded real double sequences $x,y$, cf. also \cite[Lemma 2.2]{MR1507948}:
\begin{prop}\label{prop:econvergence}
Let $x,y$ be two sequences in a locally convex first countable topological vector space such that $x-y \to_{\mathcal{I}} 0$ and $\{n: x_n \notin K\} \in \mathcal{I}$ for some compact $K$. Then $\mathrm{core}_x(\mathcal{I})=\mathrm{core}_y(\mathcal{I})$.
\end{prop}
\begin{proof}
By \cite[Lemma 3.5.i]{LMxyz} we have $\Gamma_x(\mathcal{I})=\Gamma_y(\mathcal{I})$. The conclusion follows by Theorem \ref{thm:charactcore}.
\end{proof}

%%%%%%%%%%%%%%%%%%%%%%%%%%%%%%%%%%%%%%%%%%%%%%%%%%%%%%%%%%%%%%%%%%%%

\subsection{Euler $r$-core} Fix $r \in \mathbf{R}$ and define the infinite lower triangular matrix $A=(a_{n,k}: n,k \ge 0)$ by 
$$
a_{n,k}:=\binom{n}{k}(1-r)^{n-k}r^k.
$$
if $n\ge k$ and $a_{n,k}:=0$ otherwise. The \emph{Euler} $r$\emph{-core} of a complex-valued sequence $x$, denoted by $\mathrm{E}\text{-}\mathrm{core}(x)$, has been defined in \cite[Definition 2]{MR2657795} as the least closed convex hull containing $\{y_n,y_{n+1},\ldots\}$ for all $n \in \omega$, where $y_n:=\sum_{k=0}^n a_{n,k}x_k$ for all $n$. In other words, this is the Knopp core of $Ax$, that is, 
$$
\mathrm{E}\text{-}\mathrm{core}(x)=\mathrm{core}_{Ax}(\mathrm{Fin}).
$$
In this context, the analogue of \eqref{eq:firstcharacterizationold} has been shown in \cite[Theorem 18]{MR2657795}; hence it is a special case of Theorem \ref{thm:anothercharactIcoremain}.

\section*{Acknowledgments}
The author is grateful to an anonymous referee for suggestions that helped improving the overall presentation of the paper.
%

%%%%%%%%%%%%%%%%%%%%%%%%%%%%%%%%%%%%%%%%%%%%%%%%%%%%%%%%%%%%%%%%%%%%%%%%%%%%%%%%%%%%%%%%%
%%%%%%%%%%%%%%%%%%%%%%%%%%%%%%%%%%%%%%%%%%%%%%%%%%%%%%%%%%%%%%%%%%%%%%%
%\nocite{*}
%\bibliographystyle{amsplain}
%\bibliography{ideale}

\end{document}